\documentclass{amsart}
\usepackage{mathabx}
\usepackage{amsmath}
\usepackage{mathrsfs}
\usepackage[usenames,dvipsnames]{color}
\usepackage{hyperref}
\hypersetup{
 colorlinks,
 linkcolor={blue!90!black},
 citecolor={red!80!black},
 urlcolor={blue!50!black}
}
\usepackage{tikz-cd}
\usepackage{graphicx}
\usepackage{cite}
\usepackage{amsfonts}

\newtheorem{thm}{Theorem}[section]
\newtheorem{lemma}[thm]{Lemma}

\newtheorem{prop}[thm]{Proposition}

\theoremstyle{definition}

\newtheorem{remark}[thm]{Remark}

\newtheorem{definitions}[thm]{Definitions}

\numberwithin{equation}{section}

\newcommand{\Hilb}{\mathrm{Hilb}}
\newcommand{\Gr}{\mathbf{Gr}}
\newcommand{\rank}{\mathrm{rank}}
\newcommand{\sat}{\mathrm{sat}}

\newcommand{\W}{\bigwedge}
\newcommand{\w}{\wedge}

\renewcommand{\P}{\mathbf{P}}
\newcommand{\N}{\mathbf{N}}
\newcommand{\Z}{\mathbf{Z}}
\newcommand{\kk}{{\mathbf{k}}}

\newcommand{\HH}{{\mathcal{H}}}
\newcommand{\h}{{\mathfrak{h}}}

\newtheorem{manualtheoreminner}{Theorem}

\makeatletter
\@namedef{subjclassname@2020}{
  \textup{2020} Mathematics Subject Classification}
\makeatother

\begin{document}
\author[R. Ramkumar, A. Sammartano]{Ritvik~Ramkumar and Alessio~Sammartano}
\address{(Ritvik Ramkumar) Department of Mathematics\\University of California at Berkeley\\Berkeley, CA\\94720\\USA}
\email{ritvik@math.berkeley.edu}
\address{(Alessio Sammartano) Dipartimento di Matematica \\ Politecnico di Milano \\ Milan \\ Italy}
\email{alessio.sammartano@polimi.it}


\title{On the smoothness of lexicographic points on Hilbert schemes}

\begin{abstract}
We study  the geometry of standard graded Hilbert schemes of polynomial rings and exterior algebras.
Our investigation is motivated by a famous theorem of  
  Reeves and Stillman 
for the  Grothendieck Hilbert scheme,
which states that the lexicographic point is smooth.
By contrast, 
we show that,
in standard graded Hilbert schemes of polynomial rings and exterior algebras,
 the lexicographic point can be singular, and it can lie in multiple irreducible components.
We answer questions of Peeva--Stillman and of Maclagan--Smith.

\end{abstract}

\maketitle

\section{Introduction}

Hilbert schemes are fundamental parameter spaces in algebraic geometry.
The  classic example is 
the  Grothendieck Hilbert scheme $\Hilb^{p}(\P^n)$ \cite{Grothendieck}, 
parametrizing the closed subschemes of $\P^n$ with a fixed Hilbert polynomial $p$. 
However, it is  useful to consider the more general standard graded Hilbert scheme  $\HH^\h(R)$, 
parametrizing the homogeneous ideals  with a fixed Hilbert function $\h$ in a graded ring $R$;
see  \cite{HaimanSturmfels} for the general theory and several applications.
Then $\Hilb^{p}(\P^n)$ is  a special case of $\HH^\h(R)$,
where $R$ is a polynomial ring and $\h$ is a sufficiently large truncation of $p$. 
Typical applications of  standard graded Hilbert schemes include  the study of various loci in the Grothendieck Hilbert scheme \cite{CEVV,DJNT,Erman},
 and, more recently,
  the study of ranks of forms \cite{BB,HMV}.

To aid in the study of Hilbert schemes, 
it is  beneficial to identify distinguished points on them.
Lexicographic ideals were introduced by Macaulay \cite{Macaulay} 
 to classify Hilbert functions and Hilbert polynomials. 
Several notable classes of Hilbert schemes,
including $\Hilb^p(\P^n)$ and  $\HH^\h(R)$ where $R$ is a polynomial ring or an  exterior algebra,
 possess a unique lexicographic point.
Often,
the lexicographic point can be used to obtain  geometric and algebraic information on the Hilbert scheme,
for example about connectedness \cite{Hartshorne,PeevaStillmanConnected},  irreducible components \cite{Reeves},
and syzygies  \cite{GMP,MuraiPeeva}.
A fundamental result in this context was proved by Reeves and Stillman \cite{ReevesStillman}:
the lexicographic point on the Grothendieck Hilbert scheme  $\Hilb^{p}(\P^n)$ is always smooth.
The result is particularly strong, since  $\Hilb^p(\P^n)$ is smooth if $n\leq 2$ \cite{Fogarty},
but otherwise it can have arbitrary singularities \cite{Vakil}.
An important consequence of the Reeves--Stillman theorem  is the fact that it determines a 
canonical component of $\Hilb^{p}(\P^n)$,  known as the lexicographic component.
The theorem is useful in various situations,
e.g.\ in questions of smoothness   \cite{SkjelnesSmith,Staal,RamkumarFew},
 rationality \cite{LellaRoggero}, 
 and in explicit constructions \cite{Gotzmann,RamkumarLinear}.
In some sense, the Reeves--Stillman theorem 
and Hartshorne's connectedness theorem 
are the main general tools available in the complicated study of the geography of $\Hilb^{p}(\P^n)$.
A version of the Reeves--Stillman theorem holds for toric Hilbert schemes, where the role of the lexicographic point is played by the toric point
\cite{PeevaStillmanToric},
and for multigraded Hilbert schemes in two variables, 
where the role of the lexicographic point is played by lex-most points \cite{MaclaganSmith}.

It is natural to ask whether the Reeves--Stillman theorem holds for standard graded Hilbert schemes 
$\HH^\h(R)$, 
see \cite[p.\ 157]{GMP}. 
The most prominent cases of interest are that  of the  polynomial ring $R=S$, 
which provides the most natural generalization of $\Hilb^p(\P^n)$,
and the case of  the exterior algebra $R=E$,
where the tangent space enjoys extra structure  in terms of Gr\"obner flips \cite{PeevaStillmanExterior}.
More generally, 
one would like to know whether the lexicographic point 
establishes a canonical component of $\HH^\h(R)$.
In this paper we answer these questions negatively.
For the polynomial ring, we prove in Section \ref{SectionPolynomialRing} the following result,
which settles a question of 
\cite[p.\ 1610]{MaclaganSmith}.

\newtheorem*{thm:TheoremPolynomial}{Theorem \ref{TheoremPolynomial}}
\begin{thm:TheoremPolynomial}
Let $S = \kk[x,y,z]$ and  $\h = (1, 3, 4, 4, 3, 3, 3, \ldots)$.
The standard graded  Hilbert scheme $\HH^\h(S)$ 
 is the union of two irreducible components  of dimension 8,
and  the lexicographic point  lies in their intersection;
in particular, the lexicographic point is singular.
\end{thm:TheoremPolynomial}
\noindent 
Moreover,  in Section \ref{SectionSecondPolynomialRing}
we show the existence of  standard graded Hilbert schemes
whose lexicographic point is  not even Cohen--Macaulay, 
and we also show that this pathological behavior may occur even when the lexicographic point is a limit of points parametrizing reduced subschemes. 

For the exterior algebra, we prove in Section \ref{SectionExteriorAlgebra}  the following result, 
which settles  a question of 
 \cite[p.\ 546]{PeevaStillmanExterior}.

\newtheorem*{thm:TheoremExterior}{Theorem \ref{TheoremExterior}}
\begin{thm:TheoremExterior} 
Let $E = \W \kk^5$ and $\h =(1,5,7,2)$. 
The standard graded Hilbert scheme  $\HH^\h(E)$  
 is the union of two irreducible components of dimension 14 and 15,
 and the lexicographic point  lies in their intersection;
in particular, the lexicographic point is singular.
\end{thm:TheoremExterior}

\section{Preliminaries}

In this section we fix some notation and terminology.
Throughout the paper $\kk$ denotes an algebraically closed field with $\mathrm{char}(\kk)\ne 2$.
All rings and  ideals in this paper are $\N$-graded $\kk$-vector spaces, and we will  omit the word ``graded''.
We only consider algebras $R$ that are finitely generated in degree 1, 
mainly the polynomial ring $S$ and the exterior algebra $E$, and their quotients.

We denote by $V_d$ the graded component of degree $d\in \Z$ of a vector space $V$.
We use $\langle f_1, \ldots, f_m\rangle$ to denote the $\kk$-linear span of elements $f_1,\ldots, f_m$.

The Hilbert function of an algebra $R$ is the sequence $(\dim_\kk R_d \, : \, d \in \N)$.
We only consider Hilbert functions of \emph{algebras},
therefore, by  abuse of terminology, when $I$ is an ideal,
the phrase ``Hilbert function of $I$'' refers to the Hilbert function of the quotient algebra  by $I$.

We denote by $\HH^\h(R)$ 
the {\bf standard graded Hilbert scheme} parametrizing  ideals of  $R$ with Hilbert function  $\h$
\cite{HaimanSturmfels}.
We denote by 
$\mathrm{Hilb}^{p}(\P^n)$
the {\bf Grothendieck Hilbert scheme}  parametrizing closed subschemes of $\P^n$ with Hilbert polynomial  $p$,
equivalently, saturated ideals with Hilbert polynomial  $p$.

When $R$ is a quotient of  $S$,
the $h$-vector $(h_0, h_1, \ldots, h_s)$ of $R$ is the list of coefficients of the numerator of the  Hilbert series of $R$ in its reduced rational form. We have $\deg(R)= \sum_{i=0}^s h_i$.
Since the Hilbert function of $R$ is determined by  its $h$-vector and Krull dimension,
 $h$-vectors provide a convenient notation to deal with non-Artinian algebras.

We fix the lexicographic order among monomials of $S$ or $E$.
A {\bf lexicographic ideal} is a monomial ideal $L$ such that each $L_d$ 
is spanned by the $\dim_\kk L_d$ largest monomials of degree $d$.
Classical results of Macaulay and Kruskal--Katona state that, 
for all $\h$ such that $\HH^\h(R)\ne \emptyset$, there exists a unique lexicographic ideal $L \in \HH^\h(R)$,
for both $R = S$ and $R= E$.
See \cite{GMP} for more details. 

We let  $\mathrm{in}_{\mathrm{lex}}(I)$ denote
the initial ideal of an ideal $I$.
There is a one-parameter family whose general fiber is $I$ and whose special fiber is $\mathrm{in}_{\mathrm{lex}}(I)$;
this phenomenon is known as Gr\"obner degeneration.

We denote by  $\Gr(r,V)$, respectively $\Gr(r,n)$, the Grassmannian variety  parametrizing  $r$-dimensional subspaces of $V$, respectively of $\kk^n$.
Recall that $\dim \Gr(r,n)=r(n-r)$.

A {\bf pencil} of quadrics of $R$ is a $2$-dimensional subspace $V \subseteq R_2$.

\section{The standard graded Hilbert scheme of the polynomial ring}\label{SectionPolynomialRing}

This section is devoted to the proof of Theorem \ref{TheoremPolynomial}.
Let $S = \kk[x,y,z] $ be  the polynomial ring in 3 variables over $\kk$.
We consider the standard graded Hilbert scheme 
$
\HH = \HH^{\h}(S),
$ 
parametrizing the ideals of $S$ with Hilbert function  $$
\h = (1,3,4,4,3,3,3, \ldots)
$$ 
in other words ideals $I\subseteq S$ such that
$$
 I_0 = I_1 = 0, \quad 
\dim_\kk I_2 = 2, \quad
 \dim_\kk I_3 = 6, \quad 
\dim_\kk I_d ={d+2 \choose 2} -3\, \text{ for } \, d \geq 4.
$$

Since  the ideals parametrized by $\HH$ are not saturated, 
taking saturation defines a non-trivial map $\sat: \HH \rightarrow \Hilb^3(\P^2)$.
The main idea of the proof is to understand $\HH$ by studying the fibers of this map, and induce a stratification of $\HH$ from one of $\Hilb^3(\P^2)$.
We begin by collecting basic facts about $\Hilb^3(\P^2)$.

\begin{lemma}\label{LemmaClassical3Points}
The Hilbert scheme $\Hilb^3(\P^2)$ is a smooth irreducible 6-fold. 
It is stratified by $h$-vectors into locally closed subschemes 
\begin{equation}\label{EquationStratification}
\Hilb^3(\P^2) = H_{(1,2)} \, {\textstyle \coprod}\, H_{(1,1,1)}
\end{equation} 
where  $H_{(1,2)}$ is open and  $H_{(1,1,1)}$ is an irreducible divisor.
The subscheme $H_{(1,1,1)}$ is the locus  of complete intersections of degrees $\{1, 3\}$,
while $H_{(1,2)}$ is the locus of codimension 2  ideals of minors of $2\times 3$ matrices of linear forms.
\end{lemma}

\begin{proof}
The first statement is  \cite[Theorem 2.4]{Fogarty}.
Every $J \in \Hilb^3(\P^2)$ is  saturated, hence Cohen--Macaulay, of codimension 2 and degree 3.
The only possible $h$-vectors are  $(1,2)$ and $(1,1,1)$, and
this yields the stratification \eqref{EquationStratification}.
The  ideals   $J \in H_{(1,1,1)}$ are of the form $J = (\ell, c)$ 
where $\langle \ell \rangle \in \Gr(1,S_1) \cong \P^2$ 
and $\langle c\rangle \in \Gr(1, S_3/\ell S_2) \cong \P^3$.
It follows that $H_{(1,1,1)}$ is  isomorphic to a $\P^3$-bundle on a $\P^2$,
and therefore it is a closed irreducible 5-dimensional  subscheme  of $\Hilb^3(\P^2)$.
The last statement follows from the Hilbert--Burch theorem.
\end{proof}

We define three loci in the Hilbert scheme $\HH$,
refining the stratification of $\Hilb^3(\P^2)$ by  taking into account the codimension of the quadratic part of an ideal.

\begin{definitions}\label{DefinitionsLociPolynomial}
Let $\mathcal{X}$ be the closure in $\HH$ of the locus $\mathcal{X}^{\circ}$ consisting of ideals $I $ 
such that the ideal  $(I_2)$ has codimension 2.

Let $\mathcal{X}'$ be the locus  of ideals $I \in \HH$ such that $\sat(I)$ has  $h$-vector $(1,2)$
and the ideal $(I_2)$ has codimension 1.

Let $\mathcal{Y}$ be the locus  of ideals $I \in \HH$ such that $\sat(I)$ has  $h$-vector $(1,1,1)$.
\end{definitions}

\begin{prop}\label{PropositionParametrizationXPolynomial}
The locus $\mathcal{X}\subseteq \HH$ is irreducible of dimension 8.
\end{prop}
\begin{proof}
We consider the saturation map $\mathcal{X}^{\circ} \rightarrow H_{(1,2)}$.
Let $I \in \mathcal{X}^{\circ}$ and denote  $\sat(I) = J$.
By definition, 
the ideal $(I_2)$ has codimension 2, hence it is a complete intersection of 2 quadrics.
Comparing the Hilbert functions of $I$ and $(I_2)$, we deduce that
$I = (I_2, J_4)$,
with $I_2 \in \Gr(2,J_2) \cong \P^2$.
By Lemma \ref{LemmaClassical3Points}, 
every ideal $J \in  H_{(1,2)}$ is generated by quadrics.
Moreover, by lower semicontinuity of codimension, a general $V \in \Gr(2,J_2)$ generates an ideal of codimension 2.
Thus, the fiber over each $J \in H_{(1,2)}$ is an open subset of a $\P^2$.
By Lemma \ref{LemmaClassical3Points} we conclude that $\mathcal{X}^{\circ}$, and therefore also $\mathcal{X}$,
are irreducible of dimension 8.
\end{proof}

\begin{lemma}\label{LemmaSaturated}
Every ideal $J \in H_{(1,2)}$ is in the $\mathrm{GL}_3$-orbit of one of the following
\begin{enumerate}
\item $(xy, xz, yz)$,
\item $(x^2, xy, yz) $,
\item $ (x^2, xy, xz+y^2)$,
\item $ (x^2, xy, y^2)$.
\end{enumerate}
\end{lemma}

\begin{proof}
The subscheme $V(J)\subseteq \P^2$ has dimension 0 and degree 3, 
so it is supported at 1, 2, or 3 points.
More specifically, $J = \cap_{i=1}^s \mathfrak{q}_i $
where each $\mathfrak{q}_i$ is  primary and $V(\mathfrak{q}_i)$ is supported at an isolated point, $s \leq 3$,
and $\sum_{i=1}^s \deg V(\mathfrak{q}_i) = 3$.
If $s=3$ and the  points are non-collinear, then 
we may change coordinates to have $J = (x,y) \cap (x,z) \cap (y,z) = (xy,xz,yz)$.
On the other hand, if the  points are collinear,
then $J = (\ell, c)$ with $\ell \in S_1, c \in S_3$,
thus $J$ has $h$-vector $(1,1,1)$.
If $s=2$, then we may assume $\deg V(\mathfrak{q}_1) =1$ and $\deg V(\mathfrak{q}_2) =2$.
Up to changes of coordinates, we have $\mathfrak{q}_1 = (x,y)$ and either 
$\mathfrak{q}_2 = (x,z^2)$ or $\mathfrak{q}_2 = (x^2,z)$.
However, the former case cannot occur, since $J$ contains no linear form, 
therefore
$J = (x,y) \cap (x^2,z) = (x^2, xy, yz)$.
Finally, the cases with $s=1$ follow e.g.\ from \cite[Theorem 2.1]{ManteroMcCullough}.
\end{proof}

\begin{lemma}\label{LemmaClosure}
We have $\mathcal{X}' \subseteq \mathcal{X}$.
\end{lemma}

\begin{proof}
The saturation map $\mathcal{X}' \rightarrow H_{(1,2)}$
 stratifies  $\mathcal{X}'$  by the $\mathrm{GL}_3$-orbits of  saturations.
There are four strata $\mathcal{X}'_{(1)}, \mathcal{X}'_{(2)}, \mathcal{X}'_{(3)}, \mathcal{X}'_{(4)}$,
corresponding to the  orbits of Lemma \ref{LemmaSaturated},
and it suffices to show that $\mathcal{X}'_{(i)} \subseteq \mathcal{X}$ for each $i$.
Equivalently,  it suffices  to show that,
for each of the four ideals  $J$ of Lemma \ref{LemmaSaturated} and every 
$I \in \mathcal{X}'$ with $ \sat(I)=J$,
we have $I \in \mathcal{X}$.
We also point out that for all $I \in \mathcal{X}'$, since the ideal $(I_2)$ has codimension 1, 
$I_2$ is spanned by two reducible quadrics with a common factor, 
i.e., $I_2$ is a pencil of reducible quadrics.

Stratum $\mathcal{X}'_{(1)}$:
Let $I \in \mathcal{X}'$ with $\sat(I) = J = (xy, xz, yz)$.
There are  3 pencils of reducible quadrics in $J_2$, namely $\langle xy,xz \rangle, \langle xy,yz \rangle$, and $ \langle xz,yz \rangle$,
so we may assume $I_2 = \langle xy, xz \rangle$.
Comparing the Hilbert functions of $I$ and $J$ we deduce that
$
I = \big(xy, xz, yz(\alpha y + \beta z), y^3z, y^2z^2, yz^3\big)
$
for some $\alpha, \beta \in \kk$.
In order to show that $\mathcal{X}_{(1)}'\subseteq \mathcal{X}$, 
it suffices to show that $I $ is a limit of ideals of $\mathcal{X}^{\circ}$ 
when $\alpha, \beta \ne 0$,
since $\mathcal{X}$ is closed.
We claim that a desired limit is
$$
I(t) = \big(xy+ty(\alpha y+\beta z),xz,y^3z, y^2z^2, yz^3\big) \longrightarrow (xz,xy,yz(\alpha y+\beta z),y^3z, y^2z^2, yz^3).
$$
In fact,  
the saturation of $I(t)$ is
  $\sat(I(t)) = (xy+t\alpha y^2,xz,yz)$,
which 
is the ideal of  minors of the matrix
$$
\begin{pmatrix}
x	& y & 0\\
-t\alpha y 	& y & z 
\end{pmatrix}
$$
and therefore has Hilbert function $(1,3,3,3,\ldots)$.
Comparing $I(t)$ and   $\sat(I(t))$ we see  that 
$I(t)$ has Hilbert function $\h$.
Moreover, the ideal $(I(t)_2)$ clearly has codimension 2, thus $I(t)\in \mathcal{X}^{\circ} $.

Stratum $\mathcal{X}'_{(2)}$:
Let $I \in \mathcal{X}'$ with $\sat(I) = J = (x^2, xy, yz)$.
The pencils of reducible quadrics in $J_2$ are $\langle x^2, xy\rangle $ and $\langle xy, yz\rangle $.
If  $I_2 = \langle xy, yz \rangle$
then we conclude that
$
I = \big(xy, yz, x^2(\alpha x + \beta z), x^4, x^3z, x^2z^2\big)
$
for some $\alpha, \beta \in \kk$.
When $\alpha, \beta \ne 0$,
 we have $I \in \mathcal{X}$ since,
 as in the previous paragraph,
it is a limit of ideals in $\mathcal{X}^{\circ}$:
$$
I(t) = \big(yz+tx(\alpha x+\beta z),xy,x^4, x^3z, x^2z^2\big) \longrightarrow \big(yz,xy,x^2(\alpha x+\beta z),x^4, x^3z, x^2z^2\big).
$$
If $I_2 = \langle x^2, xy \rangle$ then 
$
I = \big(x^2, xy, yz(\alpha y + \beta z), y^3z, y^2z^2, yz^3\big),
$
and the following limit for $\alpha, \beta \ne 0$ shows that $I\in \mathcal{X}$
$$
I(t) = \big(x^2+t(\alpha y+\beta z)z,xy, y^3z, y^2z^2, yz^3\big) \longrightarrow \big(x^2,xy,yz(\alpha y+\beta z), y^3z, y^2z^2, yz^3\big). 
$$

Stratum $\mathcal{X}'_{(3)}$:
Let $I \in \mathcal{X}'$ with $\sat(I) = J = (x^2, xy, xz+y^2)$.
Since $xz+y^2$ is irreducible,  the only pencil of reducible quadrics in $J_2$ is $\langle x^2, xy\rangle$.
We get
$
I = \big(x^2, xy, (xz+y^2)(\alpha y + \beta z)\big) + (xz+y^2)\big(y,z)^2
$
and,  as in the previous paragraphs,
 the following limit for $\alpha, \beta \ne 0$ shows that $I\in \mathcal{X}$
\begin{eqnarray*}
I(t) &=& \big(x^2+ t(\alpha y^2+\beta zy),xy-t(\alpha yz+\beta z^2)\big) + (y^2+xz)\big(x,y,z\big)^2 \\
&\rightarrow & \big(x^2,xy,(y^2+xz)(\alpha y+\beta z)\big)+ (y^2+xz)\big(y,z\big)^2.
\end{eqnarray*}

Stratum $\mathcal{X}'_{(4)}$:
Let $I \in \mathcal{X}'$ with $\sat(I) = J = (x^2, xy,y^2)$.
Up to changes of coordinates, we may assume $I_2 = \langle x^2, xy\rangle$,
then $I= (x^2, xy, y^2(\alpha y+\beta z),y^4, y^3z,y^2z^2) $,
and the following limit for $\alpha, \beta \ne 0$ shows that $I\in \mathcal{X}$
$$
I(t) = \big(x^2+ty(\alpha y+\beta z),xy,y^4,y^3z,y^2z^2\big) \longrightarrow  \big(x^2, xy, y^2(\alpha y+\beta z),y^4, y^3z,y^2z^2\big).
$$
\end{proof}

\begin{prop}\label{PropositionParametrizationYPolynomial}
The locus $\mathcal{Y}\subseteq \HH$ is closed and irreducible of dimension 8.
\end{prop}

\begin{proof}
We have a map $\mathcal{Y} \rightarrow H_{(1,1,1)}$ defined by $ I \mapsto \sat(I)$.
An ideal $J \in  H_{(1,1,1)}$ is of the form $J = (\ell_1, c_1)$ with $\ell_1\in S_1, c_1\in S_3\setminus (\ell_1)$.
For every $I \in \mathcal{Y}$ with $\sat(I) = J$,
by comparing Hilbert functions,
 we must have
$I = (\ell_1 \ell_2, \ell_1 \ell_3, c_2, J_4)$ where 
$$
\langle \ell_2, \ell_3 \rangle \in \Gr(2,S_1) \cong \P^2,
\quad
\langle  c_2 \rangle\in \Gr\left(1, \frac{J_3}{\langle \ell_1 \ell_2, \ell_1 \ell_3\rangle S_1 }\right) = \Gr(1,2) \cong \P^1.
$$
Thus, the fiber over each $J \in H_{(1,1,1)}$  is a $\P^1$-bundle over a $\P^2$,
and by Lemma \ref{LemmaClassical3Points} we conclude
that $\mathcal{Y}$ is irreducible of dimension 8.
Note that $\mathcal{Y}\subseteq \HH$ is closed, since it is the preimage of the closed subset $H_{(1,1,1)}
\subseteq \Hilb^3(\P^2)$.
\end{proof}

We are ready to state the main result of this section.

\begin{manualtheoreminner} \label{TheoremPolynomial}
The standard graded Hilbert scheme $\HH = \HH^\h(S)$ is a union of two  irreducible components of dimension 8.
The lexicographic point of  $\HH$ lies in the intersection of the two components, and is a singular point.
\end{manualtheoreminner}

\begin{proof}
By Definitions \ref{DefinitionsLociPolynomial} and Lemma \ref{LemmaClosure} we have $\HH = \mathcal{X}\cup \mathcal{Y}$, 
and it follows from Propositions \ref{PropositionParametrizationXPolynomial} and \ref{PropositionParametrizationYPolynomial} that
$\mathcal{X}$ and   $\mathcal{Y}$  are distinct irreducible components.

The lexicographic ideal of  $\HH$ is  $ L = (x^2, xy, xz^2, y^4, y^3z)$.
Since $\sat(L) = (x,y^3)$, we have $L \in \mathcal{Y}$.
Now consider $
J = (x^2, xy+xz-y^2) + xy (x,y,z)^2.
$ 
Its saturation  $\sat(J) =(x^2,xy+xz-y^2,xy)$ is the ideal of minors of 
$$
\begin{pmatrix} 
x & y & y + z \\
0 & x & y 
\end{pmatrix}
$$
and this shows that  $J$ has Hilbert function $\h$.
Moreover, 
$J \in \mathcal{X}^\circ$ because the two quadrics in $J$ form a regular sequence.
We have $L = \mathrm{in}_{\text{lex}}(J) $,
so by Gr\"obner degeneration we obtain
 $L \in \mathcal{X}$ as desired.
\end{proof}

\section{A second construction of a singular lexicographic point for the polynomial ring}\label{SectionSecondPolynomialRing}

In the previous section we constructed a singular standard graded Hilbert scheme starting from the Grothendieck Hilbert scheme of points which contains the first infinitesimal neighborhood of the origin in $\P^2$.
In this section, by analyzing an analogous situation  for $\P^3$,
we present a second construction of a standard graded Hilbert scheme that is even  more pathological,
since it shows that singular  lexicographic points may lie on components whose general points parametrize (saturated ideals of) reduced schemes, and
that lexicographic points may have non-Cohen--Macaulay singularities.

Let $S = \kk[x,y,z,w] $ be  the polynomial ring in 4 variables over $\kk$,
and  consider the standard graded Hilbert scheme 
$
\HH = \HH^{\h}(S),
$ 
parametrizing the ideals of $S$ with Hilbert function  $
\h = (1,4,4,4, \ldots)$.
 The lexicographic ideal of $\HH$ is   
 $$
 L = (x^2,xy,xz,xw,y^2,yz,yw^2,z^4).
 $$

As before, 
we consider the saturation map  $\sat : \HH \rightarrow \Hilb^4(\P^3)$.
It is well-known that $\Hilb^4(\P^3)$  is irreducible of dimension 12, 
and it contains  an open smooth subscheme $\mathcal{U}_1$ parametrizing  reduced collections of 4 points in $\P^3$.
Moreover, $\h$ is the ``generic'' Hilbert function in $\Hilb^4(\P^3)$,
that is, 
there is an open  subscheme $\mathcal{U}_2\subseteq \Hilb^4(\P^3)$ such that 
every $I \in \mathcal{U}_2$ has  Hilbert function $\h$.
The restriction of the saturation map to
the open subscheme $\sat^{-1}(\mathcal{U}_2)\subseteq \HH$ is the identity map, and therefore an isomorphism.
Denote by $\mathcal{X}$  the closure in $\HH$ of $\sat^{-1}(\mathcal{U}_2)$.
By the discussion above, $\mathcal{X}$   is an irreducible component of $\HH$ of dimension 12.

\begin{lemma}
The lexicographic ideal $L $ belongs to the component $\mathcal{X}$.
\end{lemma}
\begin{proof}
By \cite[Theorem 1.2]{ConcaSidman}
we have
$\mathrm{in}_{\mathrm{lex}}(I) = L$
for general $I \in \mathcal{U}_1$,
hence also for general $I \in \mathcal{U}_2$,
and this implies the desired statement.
\end{proof}

We now adopt the strategy of  \cite{Iarrobino}:
we describe  a locus too large to fit in the (smoothable) component  $\mathcal{X}$, 
yielding that $\HH$ is reducible and non-equidimensional.
As a byproduct of this description, we also see that the lexicographic ideal must belong to the intersection of at least two components. 

\begin{prop}
There exists a component $\mathcal{Y}$ of $\HH$ which contains the lexicographic ideal $L$ and has $\dim \mathcal{Y}\geq 14$.
\end{prop}
\begin{proof}
We describe a construction of ideals $I\in \HH$.
Let $\ell_1, \ell_2 \in S_1$ be  linearly independent 
and set $I'= (\ell_1, \ell_2)$;
we have $\dim_\kk I'_2 = 7$.
Let $V\subseteq I'_2$ be a subspace such that $\dim_\kk V =6$ and the ideal $(V)$ has codimension 2,
and  set $I'' = (V, I'_3)$.
Finally, let $q\in S_4$ be a quartic that is regular on $S/I'$,
and set $I= I''+(q)$.
The graded components of $I$ are
 $I_0=I_1 = 0, I_2 = V$,  and  $I_d = (I'_3, q)_d = (\ell_1, \ell_2, q)_d$ for  $d \geq 3$.
 Since $\dim_\kk S_2 = 10$ and $(\ell_1, \ell_2, q)$ has Hilbert function $(1, 2, 3, 4, 4, 4, \ldots)$,
 it follows that   $I$ has Hilbert function $\h$, so $I \in \HH$.

Denote by  $\mathcal{Z}\subseteq \HH$ be the locus of ideals of this form.
Thus,
ideals $I \in \mathcal{Z}$ are parametrized by choosing a subspace $\langle \ell_1, \ell_2 \rangle \in \Gr(2,S_1) \cong \Gr(2,4)$,
a general subspace $V \in \Gr(6, I'_2) \cong \P^6$,
and a general subspace $\langle q \rangle \in \Gr(1, S_4/I'_4) \cong \P^4$.
Moreover, all the choices are unique since $(\ell_1, \ell_2) = \sqrt{I}$, $V= I_2$, and 
$\langle q \rangle = I_4/(\ell_1,\ell_2)_4$.
This shows that the locus  $\mathcal{Z}$ is irreducible of dimension $4+6+4=14$.
Finally, we observe that $L \in \mathcal{Z}$ by choosing $\langle \ell_1, \ell_2 \rangle = \langle x, y \rangle$,
$V= \langle x^2, xy, xz, xw, y^2, yz\rangle$ and $q = z^4$.
This concludes the proof.
\end{proof}

We obtain the main result of this section.

\begin{thm}\label{TheoremRadicalPolynomial}
Let $S= \kk[x,y,z,w]$ and let $\h= (1,4,4,4,\ldots).$ 
The  standard graded  Hilbert scheme $\HH^\h(S)$  is reducible.
There exists a 12-dimensional irreducible component $\mathcal{X}$, whose general point parametrizes radical ideals,
and an irreducible  component $\mathcal{Y}$, 
which parametrizes non-saturated ideals  and has dimension at least $14$.
The lexicographic point lies in their intersection; 
in particular, the lexicographic point is singular and not Cohen--Macaulay.
\end{thm}

\section{The standard graded Hilbert scheme of the exterior algebra}\label{SectionExteriorAlgebra}

This section is devoted to the proof of Theorem \ref{TheoremExterior}.
Let $E = \W E_1 $ be the exterior algebra of a 5-dimensional vector space
 $E_1 = \langle e_1, e_2, e_3, e_4, e_5 \rangle$ over $\kk$.
We consider the standard graded Hilbert scheme 
$
\HH = \HH^{(1,5,7,2)}(E),
$ 
parametrizing the ideals of $E$ with Hilbert function  $(1,5,7,2)$, 
in other words ideals $I\subseteq E$ such that
$$
 I_0 = I_1 = 0, \quad 
\dim_\kk I_2 = 3, \quad
 \dim_\kk I_3 = 8, \quad 
 I_4 = E_4, \quad
 I_5 = E_5.
$$

We begin by reviewing some basic facts about exterior algebras.
A quadric $q\in E_2$  can be identified with a $5\times 5$ skew-symmetric matrix, and 
thus it has  even rank.
More specifically, we have 
 $\rank(q) = 2 $ if and only if $ q = \ell_1 \w \ell_2 $ 
for some $\ell_i \in E_1$ such that
$\langle \ell_1, \ell_2 \rangle = \ker(q : E_1 \rightarrow E_3 ) \in \Gr(2,E_1)$.
Likewise, we have $\rank(q) = 4 $ if and only if $ q = \ell_1 \w \ell_2 + \ell_3 \w \ell_4  $ 
for some $\ell_i\in E_1$ with $ \langle \ell_1, \ell_2, \ell_3, \ell_4 \rangle =\ker(q^2:E_1 \rightarrow E_5)\in \Gr(4,E_1)$.

\begin{lemma}\label{LemmaBasicFacts}
Let $\ell_1, \ldots, \ell_5\in E_1$, 
 let $q, q_1, q_2 \in E_2$, 
and let $V\subseteq E_1$ be a subspace.
\begin{enumerate}
\item $\langle q_1 , q_2 \rangle$ is a pencil of rank 2 quadrics $\Leftrightarrow \langle q_1 , q_2 \rangle = \langle \ell_1\w \ell_2 , \ell_1\w \ell_3 \rangle$ for some $\ell_i\in E_1 $ such that $\dim_\kk \langle \ell_1, \ell_2, \ell_3 \rangle = 3$.
\item $\rank(q) \leq 2 \Leftrightarrow q^2 = 0$.
\item Let $q = \ell_1 \w \ell_2 $ be of rank $2$, 
then $q \in \W^2 V \Leftrightarrow \ell_1, \ell_2\in V$.
\item Let $q = \ell_1 \w \ell_2 + \ell_3 \w \ell_4$ be of rank $4$, 
then $q \in \W^2 V \Leftrightarrow \ell_1, \ell_2, \ell_3, \ell_4 \in V$.
\item Let $q = \ell_1 \w \ell_2 + \ell_3 \w \ell_4$ be of rank $4$, 
then $\ell_5 \w q^2 =0  \Leftrightarrow \ell_5 \in \langle \ell_1, \ell_2, \ell_3, \ell_4 \rangle$.
\end{enumerate}
\end{lemma}
\begin{proof}
The statements follow  from the classification of quadrics discussed above.
For instance, (1) follows by inspecting the possible intersections of two subspaces $V_1, V_2 \in \Gr(2,E_1)$.
For item (4), the backward direction is obvious,
while the forward direction follows since if $q \in \W^2 V $ then $q^2 = \ell_1 \w \ell_2 \w \ell_3 \w \ell_4 \in \W^4V$.
\end{proof}

A central theme in this section is the classification of  ideals based on the existence of pencils of rank 2 quadrics.

\begin{lemma} \label{LemmaNoRank2Pencil}
Let $I\subseteq E$ be an ideal with Hilbert function $(1,5,7,2)$.
If $I$ contains no pencil of rank  $2$ quadrics,
then $I_2 \subseteq \bigwedge^2 V \subseteq E_2$ for some subspace $V \in \Gr(4,E_1)$.
\end{lemma}
\begin{proof}
Let $I_2 = \langle q_1, q_2, q_3\rangle$.
The ideal $I$ contains   rank $4$ quadrics,
hence we may assume  $q_1 = e_1 \w e_2 + e_3 \w e_4$.
We claim that $I_2 \subseteq \bigwedge^2 V$ where $ V = \langle e_1, e_2, e_3, e_4 \rangle$.

We write 
$q_2 = p_2 + \ell_2 \w e_5$, $q_3 = p_3 + \ell_3 \w e_5$ where $p_2, p_3 \in \W^2 V$ and $\ell_2, \ell_3 \in V$.
Observe that $q_1 \w V = \W^3 V \subseteq I_3$. 
Taking the image of $I_3 \subseteq E_3$ modulo $\W^3 V$  
we obtain a subspace
$ U  \w e_5 \subseteq  \W^2 V \w e_5$ such that  $U \subseteq \W^2 V$ and $\dim_\kk U = \dim_\kk I_3 - \dim_\kk \W^3 V = 8-4=4$.
Since $I_3 $ contains $e_5 \w \langle q_1, q_2, q_3\rangle$ and $V \w \langle q_2, q_3 \rangle$,
computing  images modulo  $\W^3 V$ we see that 
$U$ contains  $\langle q_1, p_2,  p_3 \rangle$ and $V \w \langle \ell_2, \ell_3 \rangle$.
We claim that $\ell_2=\ell_3=0$.
The linear forms $\ell_2, \ell_3\in V$ cannot be linearly independent, otherwise $U$ would contain the 5-dimensional subspace $V \w\langle \ell_2, \ell_3 \rangle$.
Thus $\ell_2, \ell_3$ are linearly dependent and, 
up to subtracting a multiple of $q_2$ from $q_3$ or viceversa,
we may assume that $\ell_3 =0$.
Assume by contradiction $\ell_2 \ne 0$.
Then $V \w \ell_2 $ has dimension 3.
Moreover, $q_1 \notin V \w \ell_2 $ since $q_1$ has rank 4 and thus it is not reducible. 
Since $\dim_\kk U =4$,  we conclude that $U = \langle q_1\rangle +  V\w \ell_2$,
hence  $p_2, p_3 \in   \langle q_1\rangle +  V\w \ell_2$.
We may subtract multiples of $q_1 $ from $q_2, q_3$ and assume that 
 $p_2, p_3 \in   V\w \ell_2$.
 But this means that $p_2, p_3$ are divisible by $\ell_2$, hence the same is true for $q_2, q_3$, and they form a pencil of quadrics, contradiction. 

We have proved that 
$\ell_2=\ell_3=0$, which implies  $q_2, q_3 \in \W^2 V$ as claimed.
\end{proof}

Inspired by Lemma \ref{LemmaNoRank2Pencil}, we 
define the following two loci in the Hilbert scheme.

\begin{definitions}\label{DefinitionsLociExterior}
Let $\mathcal{X}$ be the closure in $\HH$ of the locus $\mathcal{X}^{\circ}$ consisting of ideals 
$$
 I = (\ell_1 \w \ell_2, \ell_1 \w \ell_3, q)
$$
where $\ell_1, \ell_2, \ell_3 \in E_1$,
$q \in E_2$.

Let $\mathcal{Y}$ be the closure in $\HH$ of the locus $\mathcal{Y}^{\circ}$ consisting of ideals 
$$
I = (q_1,q_2,q_3,c)
$$ 
where $q_1, q_2, q_3 \in E_2$, $c\in E_3$ are such that
 $I_2 \subseteq  \W^2 V$ for some $V\in \Gr(4,E_1)$,
but $I_2 \not\subseteq  \W^2 W$ for all $W\in \Gr(3,E_1)$.
\end{definitions}

\begin{lemma}\label{LemmaOpenCondition}
Let $U \in \Gr(3, E_2)$ be such that $\dim_\kk (U \w E_1) \geq 7$.
Then $U \subseteq \W^2 W$ for some $W \in \Gr(3,E_1)$ if and only if $U^2 = 0$.
\end{lemma}

\begin{proof}
One direction is obvious: if $U \subseteq \W^2 W$ for some $W  \in \Gr(3,E_1)$,
then $U^2 \subseteq \W^4 W =0$.
Conversely, assume $U \not\subseteq \W^2 W$ for all $W \in \Gr(3,E_1)$,
it suffices to show  that $U$ contains a quadric $q$ with $\rank(q)=4$, 
 since then $q^2 \ne 0$.
Assume by contradiction $\rank(q) \leq 2$ for all $q \in U$.
Then any 2 quadrics in $U$ share a common factor, and without loss of generality $ U = \langle e_1 \w e_2, e_1 \w e_3, q \rangle$.
Since $q \notin \W^2 \langle e_1, e_2, e_3 \rangle$,  we may assume $q = \ell_1 \w e_4$.
If $\langle \ell_1 \rangle = \langle e_1 \rangle $, then $\dim_\kk(U\w E_1) =6$,
whereas if  $\langle \ell_1 \rangle \ne \langle e_1 \rangle $, then $\rank(e_1 \w \ell_2 + \ell_1 \w e_4) =4 $  for suitable $\ell_2 \in \langle e_2, e_3 \rangle$, yielding a contradiction in either case.
\end{proof}

\begin{lemma} \label{LemmaRank2Belongs} 
Let $I
\subseteq E$ be an ideal with Hilbert function $(1,5,7,2)$ containing a pencil of rank $2$ quadrics. 
Then $I$ lies in $\mathcal{X} \cup \mathcal{Y}$.
\end{lemma}
\begin{proof}
Up to changing coordinates,
we may assume  that
$
I_2 =  \langle e_1 \w e_2, e_1 \w e_3,q \rangle,
$
 $ q = e_1 \w \ell_1 + q'$,
$\ell_1 \in \{0, e_4\} $ and $q' \in \W^2 V $ 
with $V = \langle  e_2, e_3, e_4, e_5\rangle$.
Observe that the following subspaces of $I_3$ are disjoint
$$
U = \langle e_1 \w e_2 , e_1 \w e_3 \rangle \w E_1, \qquad
U' = \langle e_2 \w q', e_3 \w q', e_4 \w q', e_5 \w q \rangle.
$$
Since $\dim_\kk U =5, \dim_\kk I_3 = 8,$ we deduce $\dim_\kk U' \leq 3$.
Modulo $e_1$ this yields $\dim_\kk(q' \w V) \leq 3$,
which in turn implies $\rank(q') \leq 2.$
Write
$q' = \ell_2 \w \ell_3$ where  $\ell_2, \ell_3 \in V$.
We distinguish three cases.

Case 1: Suppose first that $\ell_1 = 0$.
Then $q' \ne 0$, and  consider the subspace $W = \langle e_1, e_2, e_3, \ell_2, \ell_3\rangle \subseteq E_1$.
If $\dim_\kk W= 5$, then we may assume $\ell_2 = e_4, \ell_3= e_5$.
In this case the ideal $(I_2)$ has Hilbert function $(1,5,7,2)$, 
so $I=(I_2) \in \mathcal{X}^{\circ} \subseteq \mathcal{X}$.
 If $\dim_\kk W \leq  4$, then
we may assume $W \subseteq \langle e_1, e_2, e_3, e_4 \rangle$.
Thus $\ell_2 , \ell_3 \in \langle e_2, e_3, e_4 \rangle$,
and we may assume 
$\ell_2 \in \langle e_2, e_3\rangle$.
Changing coordinates in $\langle e_2, e_3\rangle$, we may further assume
 $\ell_2 = e_2$,
so that $q' = e_2 \w (\alpha e_3 + \beta e_4)$
for some $\alpha, \beta \in \kk$ with $(\alpha,\beta) \ne (0,0)$.
The ideal   $(I_2)$ has Hilbert function $(1,5,7,3)$ for every $\alpha, \beta$,
 so  $I = (I_2, c)$ for some $c \in E_3$.
 If $\beta \ne 0$, then, using Lemma \ref{LemmaBasicFacts} (3), we see  that $ I \in \mathcal{Y}^{\circ} \subseteq \mathcal{Y}$.
Taking a limit $\beta \rightarrow 0$, we deduce that $ I \in \mathcal{Y}$ also when $\beta = 0$.

 Case 2:
Suppose now that $\ell_1 = e_4$ and $q' \ne 0$.
If $e_5$ appears in $q'$, then it follows 
that $\dim_\kk U' \geq 3$, 
forcing $I_3 = U \oplus U'$, so   $I = (I_2)$ has Hilbert function $(1,5,7,2)$
and   $I \in \mathcal{X}^{\circ} \subseteq \mathcal{X}$.
If $e_5$ does not appear in $q'$, then 
 $q' \in \W^2 \langle e_2, e_3, e_4 \rangle$. 
 We get
$
I_2 \w E_1 \equiv q \w E_1 \equiv \langle e_2 \w e_3 \w e_4, q \w e_5 \rangle \pmod{U},
$
so $(I_2)$ has Hilbert function 
$(1,5,7,3)$, 
and $I = (I_2, c)$ for some $c \in E_3$.
 Using Lemma \ref{LemmaBasicFacts} (3), (4),  we verify that
  $I\in\mathcal{Y}^{\circ} \subseteq \mathcal{Y}$. 

Case 3:
Suppose, finally, that  $q'=0$,
so $\ell_1 =e_4$ and 
 $I_2 = \langle e_1 \w e_2, e_1 \w e_3, e_1 \w e_4\rangle$.
 Since $(I_2)$  has Hilbert function $(1,5,7,4,1)$, it follows that
$I = (I_2,c_1,c_2)$ with $c_1,c_2 \in E_3$.
We may assume $c_1 \in \langle e_2 \w e_3 \w e_5, e_2 \w e_4 \w e_5, e_3 \w e_4 \w e_5 \rangle$,
by possibly using $c_2$ to cancel $e_2 \w e_3 \w e_4$.
Hence, $c_1$ is the product of $e_5$ and a (reducible) quadric in $\W^2 \langle e_2, e_3, e_4 \rangle$,
and we may assume $c_1 = e_2 \w e_3 \w e_5$.
By the same argument, we can choose $c_2 =e_4\w \ell_4 \w \ell_5$ with $\ell_4, \ell_5 \in  \langle e_2, e_3, e_5 \rangle$.
If $\langle \ell_4, \ell_5 \rangle  = \langle e_2, e_3 \rangle$ then $\langle c_2\rangle = \langle e_2 \w e_3 \w e_4\rangle$.
Otherwise, we may assume $\ell_5 = e_5 + \ell_6$ with $\ell_4, \ell_6 \in \langle e_2, e_3 \rangle$.
Applying the change of coordinates $ e_5 \mapsto e_5 - \ell_6$, 
and then an appropriate change of coordinates in $\langle e_2, e_3 \rangle$,
we can fix $I_2$ and $c_1$, while reducing
 $\ell_4 \w \ell_5 $ to $e_2 \w e_5$, so that $\langle c_2\rangle = \langle e_2 \w e_4 \w e_5\rangle$.
To summarize, 
when $q'=0$ we may change coordinates to transform $I$ to one of the two ideals
\begin{eqnarray*}
K &=&( e_1 \w e_2, e_1 \w e_3, e_1 \w e_4,  e_2 \w e_3 \w e_5, e_2 \w e_4 \w e_5),\\
L &=& ( e_1 \w e_2, e_1 \w e_3, e_1 \w e_4,  e_2 \w e_3 \w e_4, e_2 \w e_3 \w e_5).
\end{eqnarray*}
In order to conclude, it suffices to show that  $K\in \mathcal{X}$ and $L\in \mathcal{Y}$.
They are initial ideals
$K = \mathrm{in}_{\mathrm{lex}}(K')$ and $L = \mathrm{in}_{\mathrm{lex}}(L')$ of the ideals
\begin{eqnarray*}
K' &=& ( e_1 \w e_2, e_1 \w e_3, e_1 \w e_4 +  e_2  \w e_5),\\
L' &=&  ( e_1 \w e_2, e_1 \w e_3, e_1 \w e_4 +  e_2 \w e_3, e_2 \w e_3 \w e_5),
\end{eqnarray*}
Since $K' \in \mathcal{X}^{\circ}$ and $L' \in \mathcal{Y}^{\circ}$,
we obtain $K\in \mathcal{X}$ and $L\in \mathcal{Y}$ as desired.
\end{proof}

\begin{lemma} \label{LemmaQuadricsInSubspace}
Let $I\subseteq E$ be an ideal with Hilbert function $(1,5,7,2)$ 
such that  $I_2 \subseteq \W^2 V$ for some  $V \in \Gr(4,E_1)$.
Then $I$ lies in $\mathcal{X}\cup\mathcal{Y}$.
\end{lemma}
\begin{proof} 
If $I_2$ contains no rank $4$ quadric, then it  contains a pencil of rank 2 quadrics and the conclusion follows from Lemma \ref{LemmaRank2Belongs}.
Without loss of generality,
 we may assume $V = \langle e_1, e_2, e_3, e_4 \rangle$ and $I_2 = \langle e_1 \w e_2 + e_3 \w e_4, q_1, q_2 \rangle$ with 
 $q_1, q_2 \in \W^2 V$.
 Observe that 
 $$
\textstyle{ I_2 \w E_1 = \Big(\W^3 V\Big) \oplus \Big(I_2 \w e_5\Big)}
 $$
has dimension 7.
Since $\dim_\kk I_3 = 8$, it follows that $I = (I_2,c)$ for some $c \in E_3$, 
and 
 we conclude
 $I \in \mathcal{Y}^{\circ} \subseteq \mathcal{Y}$.
\end{proof}

Now we turn to parametrizing the loci $\mathcal{X}$ and $\mathcal{Y}$.

\begin{prop}\label{PropositionParametrizationX}
The locus $\mathcal{X}\subseteq \HH$ is irreducible and 14-dimensional.
\end{prop}

\begin{proof}
Consider the general member $I \in \mathcal{X}$,
that is, 
an ideal $I \in \mathcal{X}^{\circ}$, which satisfies
$$
I = ( \ell_1 \w \ell_2, \ell_1 \w \ell_3, q) 
$$
with Hilbert function $(1,5,7,2)$,
such that $\ell_1, \ell_2, \ell_3 \in E_1$,
$q \in E_2$.
We 
observe  that
 $\ell_1 \w q^2 =0$. 
 In fact, if  $\ell_1 \w q^2 \ne 0$,
 then by Lemma \ref{LemmaBasicFacts} (2), (5), $\rank(q)=4$ and $q\in \W^2 V$ for some $V\in \Gr(4,E_1)$ with $\ell_1 \notin V$.
However, the discussion in the first paragraph of the proof of Lemma \ref{LemmaRank2Belongs} now  generates a contradiction.

We parametrize $I$ by choosing subspaces
\begin{eqnarray}
\label{Eq1}\langle \ell_1 \rangle &\in &  \Gr(1,E_1) \cong \P^4,\\
\label{Eq2}\langle \ell_2, \ell_3 \rangle &\in &  \Gr\left(2, \frac{E_1}{\langle \ell_1\rangle}\right) \cong \Gr(2,4),\\
\label{Eq3}\langle q \rangle & \in &  \mathcal{Z}\subseteq \Gr\Big(1, \frac{E_2}{\langle \ell_1\w \ell_2, \ell_1 \w \ell_3\rangle}\Big) \cong \P^7,
\end{eqnarray}
where $\mathcal{Z}\subseteq \P^7$ is the locus of points $\langle q \rangle $ such that 
 $ \ell_1 \w q^2 =0$ and 
 \begin{equation}\label{EquationSumDimension}
 \dim_\kk \big( \langle \ell_1 \w \ell_2, \ell_1 \w \ell_3, q \rangle \w E_1 \big) = 8. 
\end{equation} 
\noindent
 Extending $\{\ell_1,\ell_2, \ell_3\}$ to a basis $\{\ell_1,\ell_2, \ell_3, \ell_4, \ell_5\}$ of $E_1$,
a system of projective coordinates in this $\P^7$ is given by the coefficients $\lambda_{i,j}$ of the non-zero basis vectors $\ell_i\w \ell_j$.
It is easy to see that the condition $\ell_1\w q^2 =0$ defines an irreducible quadric hypersurface  $\mathcal{Q}\subseteq\P^7$ with (Pl\"ucker) equation
$$
\lambda_{2,3}\lambda_{4,5}-\lambda_{2,4}\lambda_{3,5}+\lambda_{2,5}\lambda_{3,4} =0,
$$
and $\mathcal{Z}\subseteq \mathcal{Q}$ is the  subset where
\eqref{EquationSumDimension} holds. 
We claim that condition \eqref{EquationSumDimension} is open in $\mathcal{Q}$, 
equivalently, that $8$ is the largest possible dimension for the vector space in
\eqref{EquationSumDimension} when $\langle q \rangle \in \mathcal{Q}$.
If $q^2 = 0$, then $q$ is reducible, and, 
up to changing coordinates,
the space $\langle \ell_1 \w \ell_2, \ell_1 \w \ell_3, q \rangle$ is generated by monomials;
 it is easy then to conclude that 
$
  \dim_\kk \big( \langle \ell_1 \w \ell_2, \ell_1 \w \ell_3, q \rangle \w E_1 \big) \leq  8. 
$
If $q^2 \ne 0$, then, 
since $\ell_1 \w q^2 = 0$,
we have $q = \ell_1 \w \ell_6 + \ell_7 \w \ell_8$ for some $\ell_6,\ell_7,\ell_8 \in E_1$ by Lemma \ref{LemmaBasicFacts} (5),
and 
$$
\dim_\kk \frac{(\langle \ell_1 \w \ell_2, \ell_1 \w \ell_3, q \rangle \w E_1) + (\ell_1 \w  E_2)}{(\ell_1 \w  E_2)}  = \dim_\kk \frac{(\ell_7 \w \ell_8 \w E_1) + (\ell_1 \w  E_2)}{(\ell_1 \w  E_2)}  =  2
$$
which implies the desired statement, since $\dim_\kk (\ell_1 \w  E_2) = 6$.
We conclude that $\mathcal{Z}\subseteq \mathcal{Q}$ is an  open subset.
It is also  non-empty, since e.g.\
$\langle \ell_1 \w \ell_4 + \ell_2 \w \ell_5 \rangle \in \mathcal{Z}$,
therefore  $\mathcal{Z}$  is  irreducible of dimension 6.

By definition, all ideals $I \in \mathcal{X}^{\circ}$ arise in this way.
Conversely, any such choices
\eqref{Eq1}, \eqref{Eq2}, \eqref{Eq3}
determine an ideal in  $ \mathcal{X}^{\circ}$, 
since, by construction, the resulting ideal $ ( \ell_1 \w \ell_2 , \ell_1 \w \ell_3, q)$
has the correct Hilbert function $(1,5,7,2)$.
Moreover,  we claim that  each  ideal $I\in \mathcal{X}^{\circ}$ is obtained for a unique choice of subspaces
\eqref{Eq1}, \eqref{Eq2}, \eqref{Eq3}.
First, since $I=(I_2)$, 
we observe that  $I$ contains a unique pencil of rank 2 quadrics:
 otherwise, up to changes of coordinates, we would have 
 $I = ( \ell_1 \w \ell_2, \ell_1\w \ell_3, \ell_1 \w \ell_4)$
 or 
 $I = ( \ell_1 \w \ell_2, \ell_1\w \ell_3, \ell_2 \w \ell_4)$,
 and neither ideal has Hilbert function $(1,5,7,2)$.
 Given the uniqueness of the rank 2 pencil, the subspaces \eqref{Eq1} and  \eqref{Eq2} are uniquely determined by $I$.
In turn, it is obvious that the subspace \eqref{Eq3} is  uniquely determined by $I$.

We have thus constructed  an irreducible parametrization of $\mathcal{X}^{\circ}$ of dimension $ 4 + 2(4-2)+6 = 14$,
so its closure  $\mathcal{X}$ is also irreducible and $14$-dimensional.
\end{proof}

\begin{prop}\label{PropositionParametrizationY}
The locus $\mathcal{Y}\subseteq \HH$ is irreducible and 15-dimensional.
\end{prop}

\begin{proof}
Consider the general member $I \in \mathcal{Y}$,
that is, 
an ideal $I \in \mathcal{Y}^{\circ}$, which satisfies
$$
I = ( q_1,q_2,  q_3, c) 
$$
with Hilbert function $(1,5,7,2)$,
such that  $q_1, q_2, q_3 \in E_2$, $c\in E_3$,
$ I_2 \subseteq  \W^2 V$ for some $V\in \Gr(4,E_1)$,
and $I_2 \not\subseteq  \W^2 W$ for any $W\in \Gr(3,E_1)$.
We observe that $c$ is a minimal generator,
since, writing $E_1 = V \oplus \langle\ell_5\rangle$, we have
\begin{equation}\label{Equation7Cubics}
{\textstyle\dim_\kk(I_2 \w E_1) \leq \dim_\kk\Big(\W^3 V\Big) + \dim_\kk (I_2 \w \ell_5) = 4+ 3 = 7.}
\end{equation}

We parametrize $I$ by choosing subspaces
\begin{eqnarray}
\label{Eq4}V&\in & \Gr(4,E_1) \cong \P^4,\\
\label{Eq5}
 \langle q_1, q_2, q_3 \rangle &\in & \mathcal{U}\subseteq  {\textstyle\Gr\Big(3, \W^2 V\Big) }\cong \Gr(3,6),\\
\label{Eq6}\langle c \rangle & \in & \Gr\Big(1, \frac{E_3}{\langle q_1, q_2, q_3\rangle \w E_1}\Big) \cong \Gr(1,3) \cong \P^2,
\end{eqnarray}
where $\mathcal{U}$ is the locus of points $\langle q_1, q_2, q_3 \rangle \in \Gr(3,6)$ such that 
$\langle q_1,q_2,q_3 \rangle \not\subseteq  \W^2 W$ for every $W\in \Gr(3,E_1)$, 
and such that $\dim_\kk ( \langle q_1, q_2, q_3 \rangle \w E_1) = 7$.
We claim that $\mathcal{U}\subseteq \Gr(3,6)$  is an open subset.
We have already observed in \eqref{Equation7Cubics} that $7$ is the largest dimension for 
$\langle q_1, q_2, q_3 \rangle \w E_1$ when $\langle q_1, q_2, q_3 \rangle \in \Gr(3,6)$,
so the equation $\dim_\kk ( \langle q_1, q_2, q_3 \rangle \w E_1) = 7$ is an open condition.
By Lemma \ref{LemmaOpenCondition}, the condition  
$\langle q_1,q_2,q_3 \rangle \not\subseteq  \W^2 W$ for every $W\in \Gr(3,E_1)$ is also open.
Finally, $\mathcal{U}\ne \emptyset$ since e.g.\ $\langle \ell_1 \w \ell_2, \ell_1 \w \ell_3, \ell_2 \w \ell_4 \rangle \in \mathcal{U}$ 
where $V = \langle \ell_1, \ell_2, \ell_3, \ell_4 \rangle$.
Therefore, $\mathcal{U}$ is irreducible of dimension $3(6-3) = 9$.

By definition, all ideals $I \in \mathcal{Y}^{\circ}$ arise in this way.
Conversely, any such choices
\eqref{Eq4}, \eqref{Eq5}, \eqref{Eq6}
determine an ideal in  $ \mathcal{Y}^{\circ}$, 
since the resulting ideal $( q_1,q_2,  q_3, c) $
has the correct Hilbert function $(1,5,7,2)$ and the  conditions of $ \mathcal{Y}^{\circ}$ are satisfied.
Moreover,    each  ideal $I\in \mathcal{Y}^{\circ}$ is obtained for a unique choice of subspaces:
it is obvious that the subspaces \eqref{Eq5}, \eqref{Eq6} are uniquely determined by $I$, 
whereas 
for \eqref{Eq4} this follows from the requirement that $I_2 \not\subseteq  \W^2 W$ for any $W\in \Gr(3,E_1)$.

We have thus constructed   an irreducible parametrization of $\mathcal{Y}^{\circ}$ of dimension $ 4 + 9+2= 15$,
so its closure  $\mathcal{Y}$ is also irreducible and $15$-dimensional.
\end{proof}

We are  ready to state   the main theorem of this section.

\begin{manualtheoreminner} \label{TheoremExterior}
The standard graded Hilbert scheme 
$
\HH = \HH^{(1,5,7,2)}(E)
$ 
 is a union of two  irreducible components  of dimensions 14 and 15.
The lexicographic point of $\HH$ lies in the intersection of the two components, and is a singular point.
\end{manualtheoreminner}
\begin{proof} 
The classification of ideals in Lemmas \ref{LemmaNoRank2Pencil}, \ref{LemmaRank2Belongs}, \ref{LemmaQuadricsInSubspace} proves that $\HH = \mathcal{X} \cup \mathcal{Y}$,
whereas the parametrizations of Propositions \ref{PropositionParametrizationX}, \ref{PropositionParametrizationY} show that $\mathcal{X}, \mathcal{Y}$ are irreducible subschemes of the claimed dimensions.
Obviously $\mathcal{Y}$ is not contained in $\mathcal{X}$.
Comparing the minimal number of generators of the general member
we deduce, by lower semicontinuity, that $\mathcal{X}$ is not contained in $\mathcal{Y}$ either. 
Thus, $\mathcal{X}$ and $\mathcal{Y}$ are  two distinct irreducible components of $\HH$.

The lexicographic ideal of $\HH$ is $L = (e_1 \w e_2 , e_1 \w e_3 , e_1 \w e_4, e_2\w e_3 \w e_4, e_2 \w e_3 \w e_5)$. 
We saw in the proof of Lemma \ref{LemmaRank2Belongs} that $L$ lies in $\mathcal{Y}$.
On the other hand,
we  have $L = \mathrm{in}_{\mathrm{lex}}(L'')$ where
$$
L'' =\big( e_1\w e_2, e_1 \w e_3, e_1 \w e_4, e_2 \w e_3 \w e_5, e_2 \w e_4 \w (e_3 + e_5) \big).
$$
The change of coordinates $e_5 \mapsto e_5 - e_3$  shows that $L''$ lies in the same $\mathrm{GL}_5$-orbit as the ideal $K\in \mathcal{X}$ of the proof of
 Lemma \ref{LemmaRank2Belongs},
  so $L'' \in \mathcal{X}$, and by Gr\"obner degeneration 
$L \in \mathcal{X}$ as well.
Thus, $L$ belongs to the intersection of the two irreducible components, and in particular it is a singular point.
\end{proof}

\begin{remark}
It is possible to show that  $n=5$  is the least value of $n = \dim_\kk E_1$ for which reducible standard graded Hilbert schemes of $E = \W E_1$ exist.
Intuitively, 
this is due to the fact that non-trivial parametrizations do not involve very low or very high degrees in $E$.
In fact, from the point of view of parametrizations,
interesting ideals $I \subseteq E$ will typically have $I_0=I_1=0$, $I_n = E_n$, 
and no  minimal generators in degree $n-1$.
For this reason, 
when $n \leq 4$ there is not enough room for complicated families, 
and 
all $\HH^\h(E)$ are projective spaces, Grassmannians, or bundles over them, 
with the exception of $n = 4, \h = (1,4,5,2)$, 
where $\HH^\h(E)$ is the  locus of reducible quadrics.
\end{remark}

\subsection*{Acknowledgments}
The authors are grateful to the referees for carefully reading the manuscript and for several helpful comments and suggestions.
 Computations using Macaulay2 \cite{Macaulay2} were helpful in the preparation of this paper.

\end{document}